\documentclass[a4paper,11pt]{article}

\usepackage{amsthm}
\usepackage{amsbsy}
\usepackage{enumerate}
\usepackage{amscd}
\usepackage[dvips]{graphicx}
\usepackage{amsmath,amscd,amssymb,epsfig,subfig}
\usepackage[all]{xy}
\usepackage{verbatim}
\usepackage{graphics,graphicx}
\usepackage{tikz}
\theoremstyle{plain}
\newtheorem{theorem}{Theorem}[section]
\newtheorem*{maintheorem}{Main Theorem}
\newtheorem{lemma}{Lemma}[section]
\newtheorem{corollary}[theorem]{Corollary}
\newtheorem{proposition}{Proposition}[section]

\newtheorem*{question}{Question}
\newtheorem{definition}{Definition}[section]
\theoremstyle{remark}
\newtheorem{remark}{Remark}[section]

\newtheorem{example}{Example}[section]

\def\QEDclosed{\mbox{\rule[0pt]{1.3ex}{1.3ex}}}

\begin{document}

\date{}

\title{Incompleteness of pressure metric on Teichm\"uller space of a bordered surface}
  \author{Binbin Xu \\
\small \em Korea Institute for Advanced Study (KIAS)\\
\small \em 85 Hoegi-ro, Dongdaemun-gu, \\
\small \em 02455, Seoul, Republic of Korea\\
\small \em e-mail: {\tt xubin@kias.re.kr}\\
}

\maketitle
\begin{abstract}
We prove that the pressure metric on the Teichm\"uller space of a bordered surface is incomplete and its partial completion can be given by the moduli space of metric graphs for a fat graph associated to the same bordered surface equipped with pressure metric. As a corollary, we show that the pressure metric is not a constant multiple of the Weil-Petersson metric which is different from the closed surface case.
\end{abstract}

\tableofcontents

\section{Introduction}
The pressure metric was first introduced by McMullen \cite{Mcm} and Bridgeman \cite{Bri2} for Fuchsian representation space and quasi-Fuchsian representation space of the fundamental group of a closed surface respectively. By generalizing their work, Bridgeman, Canary, Labourie and Sambarino introduced in \cite{BCLS} the pressure metric on $\mathcal{C}(\Gamma,m)$, the space of conjugacy classes of regular irreducible convex representations of $\Gamma$ to ${\rm SL}(m,\mathbb{R})$, where $\Gamma$ is a word hyperbolic group. Meanwhile, inspired by the work of McMullen and Bridgeman, Pollicott and Sharp \cite{SP} introduced the Weil-Petersson type metric on the moduli space of metric graphs. The main ingredient of these results is the thermodynamic formalism developed by Bowen, Parry-Pollicott, Ruelle and others.

Let $S$ be an oriented hyperbolic surface of finite type. It is known that its fundamental group $\pi_1(S)$ is Gromov hyperbolic. By applying the results in \cite{BCLS}, the pressure metric is well defined for $\mathcal{C}(\pi_1(S),m)$. As a special case, when $m=2$, the space $\mathcal{C}(\pi_1(S),2)$ contains the Teichm\"uller space $\mathcal{T}(S)$ of $S$. When $S$ is closed, by the work of Wolpert \cite{Wo4}, the pressure metric is a constant multiple of the Weil-Petersson metric on $\mathcal{T}(S)$. Therefore in this case we know a lot of properties of the pressure metric from studying the Weil-Petersson metric. 

Now let us consider the surface $S$ with non-empty boundary. One immediate question is that:
\begin{question}
	Is the pressure metric on $\mathcal{T}(S)$ still a constant multiple of the Weil-Petersson metric?
\end{question}

One object that we will use to study this problem is the fat graph. A \textbf{fat graph} is a connected finite graph such that each vertex has its valence at least $3$ and is associated with a cyclic order among the half-edges adjacent to it. The relation between fat graphs and oriented bordered surfaces is the following. With the cyclic order around each vertex, we can define oriented cycle path on the graph. More precisely, starting walking along one edge, we always turn to the next edge at each vertex according to the cyclic order. By gluing a cylinder to each such cycle path, we obtain an oriented bordered surface. Conversely, an oriented bordered surface associated to a fat graph can retract to this fat graph by retracting every cylinder to a circle.

\begin{remark}
We can always assume that the valence of each vertex in the fat graph is $3$. We call it a tri-valence fat graph. Those fat graphs with higher vertex valence can be obtained from a tri-valence fat graph by retracting some of its edges to points. \QEDclosed
\end{remark}
\begin{remark}
If we change the cyclic order around vertices of a fat graph, the topology type of its associated surface may change. \QEDclosed
\end{remark}
\begin{remark}
For one fixed oriented bordered surface, the dual graph of any ideal triangulation is a tri-valence fat graph where the cyclic order at each vertex is induced by the orientation of the surface. \QEDclosed
\end{remark} 

Let $\mathbb{G}$ be a finite graph with valence at each vertex bigger or equal to $3$. A metric on $\mathbb{G}$ is a positive weight of its edges. The moduli space of metrics on $\mathbb{G}$ is the space of all metrics renormalized such that the topological entropy of the associated geodesic flow on $\mathbb{G}$ is $1$. Let us denote the moduli space by $\mathcal{M}_1(\mathbb{G})$. In \cite{SP}, Pollicott and Sharp studied this space and introduced the Weil-Petersson type metric on it. In our paper, we renormalize this metric and define the pressure metric on $\mathcal{M}_1(\mathbb{G})$.

Our main result is stated as follows:
\begin{maintheorem}\label{completion}
 The pressure metric on $\mathcal{T}(S)$ is not complete. Moreover there exists a fat graph $\mathbb{G}$ for $S$, such that $\mathcal{T}(S) \sqcup\mathcal{M}_1(\mathbb{G})$ is its partial completion where $\mathcal{M}_1(\mathbb{G})$ is equipped with the pressure metric. 
\end{maintheorem}

\begin{remark}
	In our construction, we are able to glue the space $\mathcal{M}_1(\mathbb{G})$ to part of the boundary of $\mathcal{T}(S)$ such that the pressure metrics on the two spaces can be also put together to give a metric on $\mathcal{T}(S) \sqcup\mathcal{M}_1(\mathbb{G})$. This is what we mean by "partial completion". \QEDclosed
\end{remark}

The proof is from studying certain path in $\mathcal{T}(S)$ along which the hyperbolic surfaces degenerate to $\mathbb{G}$ equipped with a metric determined by the degenerating path. We show that these paths have finite lengths with respect to the pressure metric.

At the same time, by our definition of Weil-Petersson metric on $\mathcal{T}(S)$, we can see that these paths have finite lengths with respect to the Weil-Petersson metric by the work of Wolpert in \cite{Wo2}. Using the work of Mazur in \cite{Ma1}, we are able to describe its metric completion. By showing that the pressure metric and the Weil-Petersson metric have different completions, we are able to give a negative answer of the problem above:
\begin{corollary}\label{WP}
 The pressure metric is not a constant multiple of the Weil-Petersson metric on $\mathcal{T}(S)$.
\end{corollary}

\section*{Acknowledgement:}
The author would like to thank Greg McShane for introducing such an interesting question and his useful discussions and suggestions. He would like to thank Martin Bridgeman, Mark Pollicott and Richard Sharp for carefully reading the previous manuscript and giving useful comments and suggestions. He would like to thank Pierre Dehornoy, Lien-Yung Kao, Louis Funar and Andres Sambarino for their useful discussions and suggestions. The author acknowledges support from U.S. National Science Foundation grants DMS 1107452, 1107263, 1107367 "RNMS: GEometric structures And Representation varieties" (the GEAR Network).

\section{Teichm\"uller space}
Let $S=S_{g,r}$ be a surface of genus $g$ with $r$ boundary components.
\subsection{Definition of Teichm\"uller space}
A \textbf{\textit{hyperbolic structure}} $m$ on $S$ is a Riemannian metric with constant curvature $-1$. Let ${\rm Homeo}^+(S)$ denote the set of all orientation preserving homeomorphisms on $S$. This set has a group structure given by compositions of homeomorphisms. If $S$ has non-empty boundary, we only consider those hyperbolic structures such that all boundary components are simple closed geodesics, instead of punctures or cone singularities. We ask all homeomorphisms to fix each boundary component setwise. 
\begin{definition}
A \textbf{marked hyperbolic structure} on $S$ is a pair $(m,f)$ where $m$ is a hyperbolic structure on $S$ and $f\in{\rm Homeo}^+(S)$. Two such structures $(m_1,f_1)$ and $(m_2,f_2)$ are said to be \textbf{equivalent} if and only if there exists an isometry $\phi:(S,m_1)\rightarrow (S,m_2)$ such that $\phi\circ f_1$ is homotopic to $f_2$. We denote it by $(m_1,f_1)\sim(m_2,f_2)$.
\end{definition}
\begin{definition}
 The \textbf{Teichm\"uller space}
 $\mathcal{T}(S)$ of
 $S$ is defined
 to be the quotient space:
 \begin{equation*}
  \mathcal{T}(S):=
  \{\textrm{marked hyperbolic structures
  on $S$}\}/\sim.
 \end{equation*}
\end{definition}

In each class there is one representative whose marking is the identity map of $S$. We will identify the Teichmuller space $\mathcal{T}(S)$ with the space of marked hyperbolic structure $(m,f)$ where $f=\mathrm{id}$. As a convention, we will simply use the metric $m$ to denote a point in $\mathcal{T}(S)$.

\subsection{Weil-Petersson metric on $\mathcal{T}(S)$}
In this part, we first recall the Weil-Petersson metric on $\mathcal{T}(S)$ of closed surface $S$. Then we give our definition of the Weil-Petersson metric for bordered surface case. 

The original definition of Weil-Petersson metric on $\mathcal{T}(S)$ comes from the study of the moduli space of Riemann surfaces. Recall that there is a 1-1 correspondence between hyperbolic structures and complex structures on a surface. The Teichm\"uller space can also be defined as the moduli space of marked complex structures on a closed surface. When the surface equipped a complex structure, we call it a Riemann surface. The fiber of cotangent space of $\mathcal{T}(S)$ based at a Riemann surface $M$ can be identified with the space of holomorphic quadratic differentials on $M$. The Petersson pairing for holomorphic quadratic differentials induces the Weil-Petersson co-metric which moreover induces the Weil-Petersson metric on the tangent space of $\mathcal{T}(S)$.

This metric have lots of interesting properties. In particular, we are interested in the following two properties:
\begin{theorem}[Wolpert \cite{Wo2}]
	The Weil-Petersson metric on $\mathcal{T}(S)$ is not complete.
\end{theorem}
\begin{theorem}[Masur \cite{Ma1}]
	The completion of Weil-Petersson metric on $\mathcal{T}(S)$ is the augmented Teichm\"uller space equipped with the Weil-Petersson metric.
\end{theorem}
In the augmented Teichm\"uller space, the parts added to the boundary of $\mathcal{T}(S)$ are Teichm\"uller spaces of pinched surfaces obtained by pinching certain number of pairwise disjoint simple closed geodesics on $S$. In particular, the last level of completion is given by pinching a maximal number of geodesics on $S$. In this case, the geodesics where the pinching happens form a pair of pants decomposition of $S$. After pinching all these curves, we obtain a disjoint union of three punctured spheres whose Teichm\"uller space is a point. The following picture tells one way to fully pinch a genus $2$ surface:
\begin{center}
\includegraphics[scale=0.7]{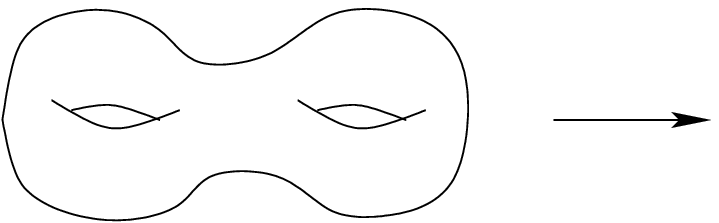}
\includegraphics[scale=0.7]{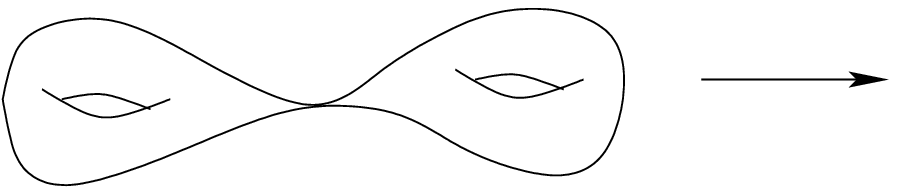}
\includegraphics[scale=0.7]{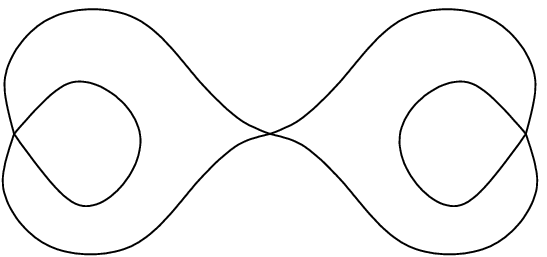}
\end{center}

For the bordered surface case, as far as we know, there is no canonical way to define the Weil-Petersson metric on $\mathcal{T}(S)$. In this paper we use the following definition. We consider the double of $S$ which is topologically a genus $(2g+r-1)$ surface $\mathrm{D}S$. Given a point in $\mathcal{T}(S)$, we double it by gluing to it the same marked hyperbolic surface with different orientation. Moreover we ask the twists at each gluing to be $0$. 
\begin{center}
 \includegraphics[scale=0.6]{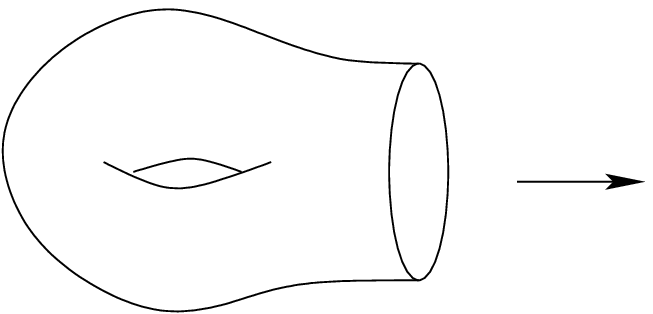}
 \includegraphics[scale=0.6]{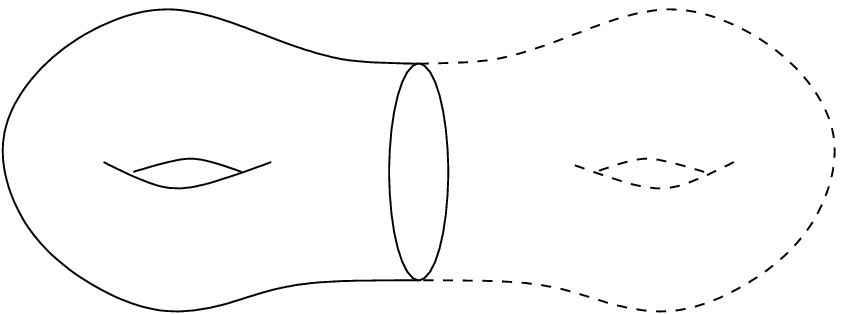}
\end{center}
In this way we found a point in $\mathcal{T}(\mathrm{D}S)$ which moreover induces a embedding of $\mathcal{T}(S)$ into $\mathcal{T}(\mathrm{D}S)$. 

\begin{definition}
 The \textbf{Weil-Petersson metric} on $\mathcal{T}(S)$ is given by pulling back the Weil-Petersson metric on $\mathcal{T}(\mathrm{D}S)$ by the embedding map.
\end{definition}

\subsection{"New" coordinate system on $\mathcal{T}(S)$}
In the following, we will give a "new" coordinates system, constructed using orthogeodesics on a bordered surface $S$. The idea of using orthogeodesics to construct coordinates system exists already in several earlier papers (see for example \cite{luo}, \cite{Ren}, \cite{Ush}). We construct the coordinates system in a different way. The reason for making such a modification is that using this new construction it will be easy to describe the path along which we study the incompleteness of pressure metric and to do the estimations.

\begin{definition}
An \textbf{orthogeodesic} $\alpha$ on $S$ is a geodesic arc hitting the boundary perpendicularly at both ends. Its length $l(\alpha)$ is an \textbf{element of orthospectrum} of $S$ and we refer to it as an \textbf{ortholength}. An orthogeodesic is said to be \textbf{simple} if it has no self-intersection.
\end{definition}

\begin{example}
The following picture shows two orthogeodesics on surface $S_{2,2}$. The orthogeodesic $\alpha$ has $1$ self-intersection, while the other one $\beta$ is simple.
\begin{center}
 \includegraphics[scale=0.8]{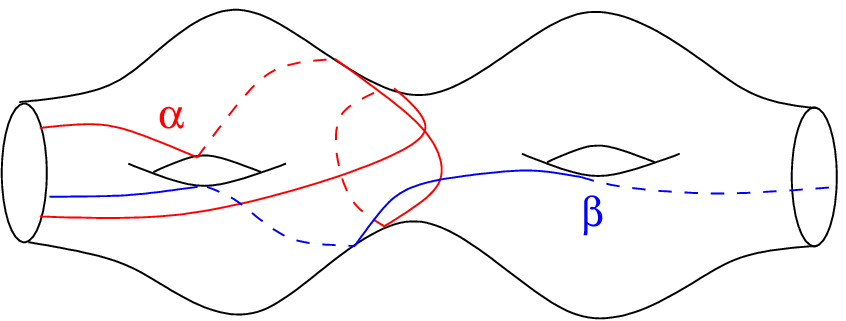}
\end{center}
\end{example}

\begin{remark}
These definitions were introduced by Basmajian in \cite{Bas} for a finite volume hyperbolic manifold with totally geodesic boundary. There he proved an identity relating the orthospectrum of a hyperbolic manifold to the volume of its boundary, which is now called Basmajian identity. \QEDclosed
\end{remark}

\begin{definition}
 An \textbf{ideal triangulation} $T$ for $S$ is a maximal collection of simple pairwise disjoint orthogeodesics of $S$.
\end{definition}

\begin{remark}
 This is a definition extended from the usual one for a punctured surface. The complement of $T$ in $S$ is, instead of a union of ideal triangles, a union of right-angled hexagons. Notice that when the lengths of all boundary components become $0$, all boundary components become punctures. The orthogeodesics become complete geodesics. The ideal triangulation become the usual one for a puncture surface. \QEDclosed
\end{remark}

The number of arcs contained in an ideal triangulation of $S$ only depends on the topology of $S$. It equals to $-3\chi(S)$ where $\chi(S)=2-2g-r$ is the Euler characteristic of $S$. The number of right-angled hexagons is then equal to $-2\chi(S)$. To simplify the notation, we set $s=-\chi(S)$.
\bigskip

We first show that there exists a special ideal triangulation for each bordered surface $S$ which we will use later in our construction.

\begin{lemma}\label{nonadjacent}
For each bordered surface, there exists an ideal triangulation such that we can choose half number of the right-angled hexagons in its complement which are pairwise non-adjacent.
\end{lemma} 
\begin{remark}
	Let $T$ be any ideal triangulation of $S$. Let us consider the dual graph $\mathbb{G}$ of $T$. In graph theory, an \textit{independent set} of a graph is a subset of vertices such that they are pairwise non-adjacent. The statement is then equivalent to the following one: there exists an ideal triangulation $T$, such that its dual graph $\mathbb{G}$ has one independent set consisting of half number of its vertices. 
\end{remark}
\begin{proof}
	To simplify the proof, all discussion will be about punctured surfaces.
	
	The proof is by induction. We first consider the case where the surface is either a once-punctured torus or a three-puncture sphere. These are the trivial cases, because there are only two ideal triangles in the complement of each ideal triangulation.
	
	Recall that any punctured surface $S_{g}^r$ of genus $g\ge1$ with $r$ punctures can be obtained by adding $r-1$ punctures to the torus and then taking its connected sum with $g-1$ torus. And an $r$-puncture sphere with $r\ge3$ can be obtained by adding $r-3$ punctures to the three-puncture sphere. Each step in the induction is given by adding a puncture or a torus to the surface.
	
	We claim that if a punctured surface $S_{g}^r$ has an ideal triangulation with the property in the lemma, so does $S_{g+1}^r$ and $S_{g}^{r+1}$.
	
	First, let us look at the case of adding one puncture. Suppose that the lemma is true for $S_{g}^r$. Then we can find an ideal triangulation $T$ and color each ideal triangle in $S\setminus T$ with either white or black such that two adjacent ideal triangle have different colors. Without loss of generality, we can assume the new puncture is added to the interior of one white ideal triangle. Then the triangulation of this once-punctured triangle together with $T$ will give an ideal triangulation $T'$ of $S_{g}^{r+1}$. The following picture provides the ideal triangulation of once-punctured triangle such that $T'$ is the one to show that the lemma is true for $S_{g}^{r+1}$.
	\begin{center}
	 \includegraphics[scale=0.8]{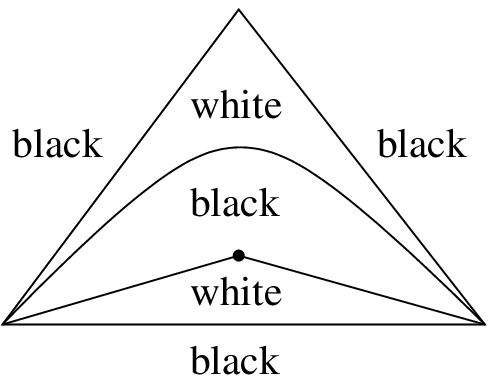}
	\end{center}
	
	The argument for $S_{g+1}^r$ is similar. We can assume that the torus is glued to the interior of a white triangle. The triangle and the torus form $\mathbb{T}$ a one-holed torus with three punctures on its boundary. Its ideal triangulation together with $T$ induces an ideal triangulation of $S_{g+1}^r$. The following picture provides the ideal triangulation of $\mathbb{T}$ such that $T'$ is the one to show that the lemma is true for $S_{g+1}^{r}$.
	
	\begin{center}
	 \includegraphics[scale=0.8]{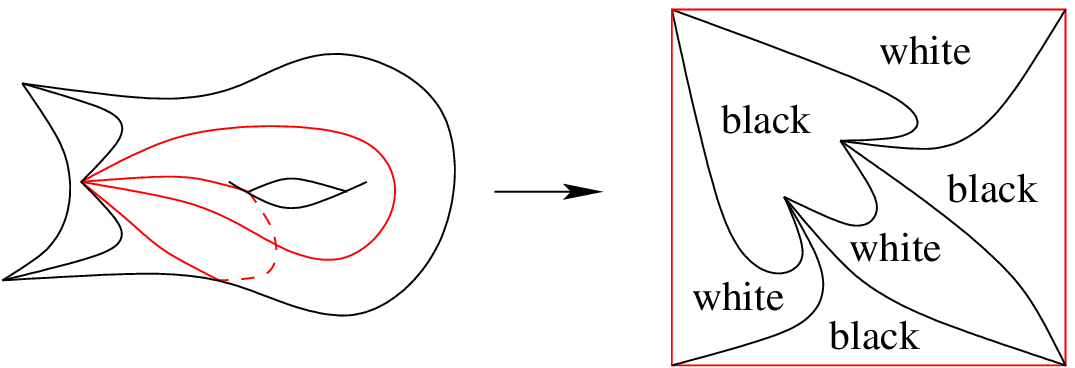}
	\end{center}
	
\end{proof}

Let $T$ denote the ideal triangulation of $S$ satisfying the property stated in the above lemma and let $\mathbb{G}$ denote its dual graph. We fix $T$ and $\mathbb{G}$ for the rest part of this paper. Denote by $\alpha_1,...,\alpha_{3s}$ the orthogeodesics in $T$. Their end points separate the boundary into $6s$ segments, denoted by $a_1,\dots,a_{6s}$. Then the following map is well defined:
\begin{equation*}
 \mathbb{O}_T: \mathcal{T}(S)\rightarrow (\mathbb{R}^+)^{6s},
\end{equation*}
where the image of $m$ is the $m$-lengths of $a_j$'s. Then we can show that:
\begin{theorem}
The map $\mathbb{O}_T$ defines an homeomorphism of $\mathcal{T}(S)$ to its image and induces a coordinates system of $\mathcal{T}(S)$.
\end{theorem}
\begin{proof}
The proof is similar to the one in \cite{Farb} where a similar result was proved for the Teichm\"uller space of the pair of pants.

Cutting $S$ along $T$, we obtain $2s$ right-angled hexagons. Each of these hexagons has three alternative sides lying on the boundary of $S$. By the choice of $T$, we can choose half number of the right-angled hexagons in $S\setminus T$ which are pairwise non-adjacent. We denote them by $H_1,...,H_s$. We consider union of their sides lying on the boundary of $S$, and denote them by $b_1,...,b_{3s}$. Other right-angled hexagons will be denoted by $H^c_1,\dots,H^c_s$. We denote by $\{b^c_1,\dots,b^c_{3s}\}$ their sides lying on the boundary of $S$. 

Let $\pi$ be the restriction of $\mathbb{R}^{6s}$ to $\mathbb{R}^{3s}$ by forgetting the lengths of $b^c_i$'s. We first show that the map $\pi\circ\mathbb{O}_T$ is bijective. Recall that the hyperbolic structure on a right-angled hexagon $H$ is determined by three of its sides lengths and the Teichm\"uller space $\mathcal{T}(H)$ of $H$ is homeomorphic to $(\mathbb{R}^{+})^3$. Given any $3s$-tuple $(b_1,\dots,b_{3s})$, it first determines the hyperbolic structures on $H_i$'s. Since all arcs in the ideal triangulation $T$ appear as sides of $H_i$'s, the $3s$-tuple $(b_1,\dots,b_{3s})$ also determine the lengths of $\alpha_i$'s, and therefore the hyperbolic structures on $H^c_i$'s. This means that a $3s$-tuple $(b_1,\dots,b_{3s})$ can determine the hyperbolic structure on $S$. Since the $3s$-tuple $(b_1,\dots,b_{3s})$ can be chosen arbitrarily, the map $\pi\circ\mathbb{O}_T$ is bijective.

The map $\mathbb{O}_T$ is a homeomorphism to its image. Because if two points are close in $\mathrm{Im}(\mathbb{O}_T)$, for each right-angled hexagon, the hyperbolic structures are nearly isometric. This is equivalent to say that the corresponding points of $\mathcal{T}(S)$ are close with respect to the topology given by identifying $\mathcal{T}(S)$ with the space of conjugacy classes of discrete faithful representations of $\pi_1(S)$ into ${\rm PSL}(2,\mathbb{R})$.

Therefore $\pi\circ\mathbb{O}_T$ is a homeomorphism and gives a coordinates system of $\mathcal{T}(S)$.
\end{proof} 

\section{The pressure metric on $\mathcal{T}(S)$}\label{S}
In this section, we recall the construction of the pressure metric on $\mathcal{T}(S)$ which is a special case of the result in \cite{BCLS}.

\subsection{Subshift space of finite type associated to $S$}
Let $\mathrm{UT}{S}$ denote the unit tangent bundle of $S$ consisting of all unit tangent vector in the tangent bundle of $S$. Let $\phi^t$ denote the geodesic flow on $\mathrm{UT}{S}$. Let $\mathrm{NW}S$ denote the non-wandering part of $\mathrm{UT}{S}$ consisting of all unit tangent vectors tangent to geodesics never hitting the boundary transversely at both ends. We can associate a subshift space of finite type to $\mathrm{NW}S$. 

The first step is to take a finite partition of $\mathrm{NW}S$. We consider the ideal triangulation $T=\{\alpha_1,...,\alpha_{3s}\}$ of $S$ chosen above. Among these arcs, we can choose a subset $\{\alpha'_1,...,\alpha'_{s+1}\}$ such that we obtain a topological disk by cutting $S$ along these arcs. The partition is given by cutting $\mathrm{NW}S$ along the fiber over $\alpha_i'$'s. 

Under this partition, each orbit in $\mathrm{NW}S$ is cut into oriented segments. The starting points of these oriented segments have $2s+2$ types corresponding to the two sides of $\{\alpha_1',...,\alpha'_{s+1}\}$. We can denote by $I=\{\alpha_1'^{\pm},...,\alpha_{s+1}'^{\pm}\}$. We associate to each oriented segment the type of its starting point, and call it the \textbf{\textit{symbol}} of this segment. Thus each orbit can be determined by a bi-infinite word whose letters are those symbols. We can label the letters in a word using integers respecting the orientation of this orbit. Each labeling is determined by its $0$-position segment. We call a labeled bi-infinite word a \textbf{\textit{state}}. It is clear that there is a one to one correspondence between states and orbit segments. If two segments are contained in a same orbit, then their corresponding states have the same bi-infinite word while their labelings of letters are different by an action of a translation of $\mathbb{Z}$.

We denote the space of all states by:
\begin{equation}
 X=\{x=(x_n)_{n\in\mathbb{Z}}\mid
 \forall n\in\mathbb{Z},\,\,x_n\in I,\,\,x_n\neq x_{n+1}^{-1}\}.
\end{equation}
It is a \textbf{\textit{subshift space of finite type}} associated with a \textbf{\textit{shift operator}} $\sigma$ defined by:
\begin{equation}
 \sigma(x)=y,
\end{equation}
such that $y_n=x_{n+1}$.

The space $X$ can be equipped with the Tychonoff topology such that $X$ is compact. The dynamical system $(X,\sigma,\mathbb{Z})$ is topologically conjugate to the geodesic flow on $\mathrm{NW}S$. The action of $\sigma$ on $X$ is mixing.

\begin{remark}
 One should think $(X,\sigma,\mathbb{Z})$ as the discrete version of $(\mathrm{NW}S,\phi^t,\mathbb{R})$.
\end{remark}

\subsection{Thermodynamic formalism}
Let $C(X,\mathbb{R})$ be the space of H\"older continuous functions on $X$ taking value in $\mathbb{R}$. 

\begin{definition}
 Two H\"older continuous functions $F_1$ and $F_2$ on $X$ are said to be \textbf{Liv{\v s}ic cohomologous} if there exists a continuous function $F_3$ on $X$ such that: $F_1-F_2=F_3-F_3\circ\sigma$. 
\end{definition}

Let $x\in X$ be contained in a periodic orbit. The minimal strictly positive integer $p$ such that $\sigma^p(x)=x$ is called the \textbf{\textit{period}} of the orbit of $x$. The \textbf{\textit{$F$-period}} of this orbit for a function $F\in C(X,\mathbb{R})$ is defined to be the following quantity:
\begin{equation}
	F(x)+F(\sigma(x))\cdots+F(\sigma^{p-1}(x)).
\end{equation}

\begin{proposition}[Liv{\v s}ic \cite{Liv}]\label{livi}
 Two function $F_1$ and $F_2$ are Liv{\v s}ic cohomologous if and only if for each periodic orbit in $X$, its $F_1$-period equals to its $F_2$-period.
\end{proposition}

\begin{definition}
The \textbf{pressure} $P(F)$ of $F$ is defined by:
\begin{equation}\label{pressure}
 P(F):=\limsup_{n\rightarrow\infty}\frac{1}{n}\log\left(\sum_{\sigma^n(x)=x}e^{F(x)+F(\sigma(x)) \cdots+F(\sigma^{n-1}(x))}\right).
\end{equation}

The map $P:C(X,\mathbb{R})\rightarrow\mathbb{R}$ sending each $F$ to its pressure is called the \textbf{pressure function}.
\end{definition}

There is another way to define the pressure of a H\"older continuous function.

Given a $\sigma$-invariant probability measure $\mu$ on $X$. Let $\gamma$ be a $\mu$-measurable finite partition of $X$. Denote by
\begin{equation*}
\bigvee_{i=0}^{n+1}(\sigma)^i\gamma
\end{equation*}
the sigma algebra generated by the collection of $(\sigma)^i\gamma$ for $0\le i\le(n+1)$. We define the following quantity:
\begin{equation*}
H(\sigma,\mu,\gamma):=-\lim_{n\rightarrow+\infty}\frac{1}{n}\left(\sum_{A\in\bigvee_{i=0}^{n+1}(\sigma)^i\gamma}\mu(A){\rm log}\,\mu(A)\right).
\end{equation*}

\begin{definition}
	The entropy of $\sigma$ with respect to $\mu$ is defined by
	\begin{equation*}
	h_\mu(\sigma)=\sup_{\gamma}H(\sigma,\mu,\gamma).
	\end{equation*}
\end{definition}

Denote by $M^\sigma$ the space of all $\sigma$-invariant probability measures on $X$. Then the pressure of a H\"older continuous function $F$ can be given by the following formula:
\begin{equation}
	P(F):=\sup_{\mu\in M^\sigma}\{h_\mu(\sigma)+\int_X F\mathrm{d}\mu\}
\end{equation}

In particular, the supreme can be achieved and the measure such that
\begin{equation}
	P(F)=h_\mu(\sigma)+\int_X F\mathrm{d}\mu
\end{equation} 
is called the \textbf{\textit{equilibrium state}} for $F$. We denote it by $\mu_F$.

\begin{remark}
	The pressure only depends on the Liv{\v s}ic cohomology class of $F$.
	
	Each H\"older continuous functions $F$ has a unique equilibrium state $\mu_F$. Two H\"older continuous functions $F_1$ and $F_2$ have a same equilibrium state if and only if $F_1$ is Liv{\v s}ic cohomologous to $F_2+c$ where $c$ is a constant function on $X$. \QEDclosed
\end{remark}

\begin{proposition}[Parry-Pollicott \cite{PaPo}]
Let $F_1$ and $F_2$ be two H\"older continuous functions on $X$. Let $\mu_1$ denote the equilibrium state of $F_1$. Then,
	\begin{enumerate}
		\item The pressure function $P$ is analytic along the path $F_1+tF_2$,
		\item If $P(F_1)=0$, the derivative of the pressure function $P$
		along the path $F_1+tF_2$ at $t=0$ is given by:
		\begin{equation}
		\frac{{\rm d}P(F_1+tF_2)}{{\rm d}t}|_{t=0}=\int_X F_2\,{\rm d}\mu_1,
		\end{equation}
		\item If moreover $F_2$ satisfies that $\int_X F_2\,{\rm d}\mu_1=0$, then we have:
		\begin{equation}
		\frac{{\rm d}^2P(F_1+tF_2)}{{\rm d}t^2}|_{t=0}={\rm Var}(F_2,\mu_1),
		\end{equation}
		where $${\rm Var}(F_2,\mu_1):=\lim\limits_{n\rightarrow\infty}\frac{1}{n}		\int_X(F_2(x)+F_2(\sigma(x))+\cdots+F_2(\sigma^{n-1}(x)))^2\,{\rm d}\mu_1,$$
		Moreover ${\rm Var}(F_2,\mu_1)=0$ if and only if $F_2$ is Liv{\v s}ic cohomologous to $0$ function.
	\end{enumerate}
\end{proposition}

Now assume that $F$ is a positive H\"older continuous function. We can define its topological entropy in the following way. Let $T$ be a positive real number. We denote by $R_T(F)$ the set of all periodic orbit in $X$ with $F$-period less than $T$. 

\begin{definition}
	The \textbf{topological entropy} $h(F)$ of $F$ is given by:
	\begin{equation*}
	h(F)=\limsup_{T\rightarrow+\infty}\frac{\log|R_T(F)|}{T}.
	\end{equation*}
\end{definition}

\begin{proposition}[Parry-Pollicott \cite{PaPo}]
	The topological entropy $h$ of $F$ satisfies: $P(-hF)=0$.
\end{proposition}

\subsection{Construction of the pressure metric}
Let $C_0(X,\mathbb{R})\subset C(X,\mathbb{R})$ denote the space of pressure $0$ functions on $X$. The \textbf{\textit{tangent space}}	of $C_0(X,\mathbb{R})$ at a function $F$ is defined to be:	\begin{equation}
	{\rm T}_F C_0(X,\mathbb{R}):=\{G\in C(X,\mathbb{R})\,:\,\int_X G\,{\rm d}\mu_F=0\}.
\end{equation}
\begin{definition}
The \textbf{\textit{pressure form}} on $C_0(X,\mathbb{R})$ is defined by the Hessian of the pressure function $P$ renormalized by a factor $$-\frac{1}{\int_X F\,{\rm d}\mu_F}.$$
\end{definition}
It induces the \textbf{\textit{pressure semi-norm}} given by the following formula:
\begin{equation}
  \|G\|^2:=-\frac{{\rm Var}(G,\mu_F)}{\int_X F\,{\rm d}\mu_F}.
 \end{equation}
for $G\in {\rm T}_F C_0(X,\mathbb{R})$.

Now consider a marked hyperbolic structure $m$ on $S$. It induce a positive H\"older continuous function $F_m$ sending each $x\in X$ to the $m$-length of the $0$-position segment associated to $x$. Then we can define the \textbf{\textit{thermodynamic map}}:
\begin{eqnarray*}
	\mathcal{I}:\mathcal{T}(S)&\rightarrow& C_0(X,\mathbb{R})\\
			m &\mapsto& -h_mF_m
\end{eqnarray*}
where $h_m$ is the topological entropy of $F_m$.

Thus the pull back of the pressure form by $\mathcal{I}$ defines a positive semi-definite $2$-form $p$ on the tangent space of $\mathcal{T}(S)$. We still call it the pressure form. In \cite{BCLS}, the authors obtained the following result:
\begin{theorem}[Bridgeman-Canary-Labourie-Sambarino \cite{BCLS}]
The pressure form $p$ on $\mathcal{T}(S)$ is non-degenerate, thus induces a Riemannian metric on $\mathcal{T}(S)$.
\end{theorem}
\begin{definition}
	This metric is called the \textbf{pressure metric}. We denote it by $g_p$.
\end{definition}

\section{The Weil-Petersson type metric on the moduli space of metric graphs}
In this section, we first recall the result about the Weil-Petersson type metric on the moduli space of metric graphs by Pollicott and Sharp in \cite{SP}, then we give the definition of the pressure metric on the same space which comes from rescaling their metric.

\subsection{Moduli space of metric graph}
Let $\mathbb{G}=(V,E)$ be a finite non-oriented graph with valence at each vertex at least $3$. Let $E=\{e_1,\dots,e_k\}$ denote the set of edges.

\begin{definition}
 A \textbf{metric} on $\mathbb{G}$ is a positive edge weighting $l:E\rightarrow\mathbb{R}^+$. The value $l(e)$ is the \textbf{length} of edge $e$. The \textbf{volume} ${\rm Vol}(l)$ of $l$ is given by the sum of all edge lengths.
\end{definition}
Let $\mathcal{M}(\mathbb{G})$ denote the space of all metrics on $\mathbb{G}$.

Let $\overline{\mathbb{G}}$ denote the oriented graph which is a branched $2$-cover of $\mathbb{G}$  branching over vertices of $\mathbb{G}$. The two lifts of an edge $e_j$ in $\overline{\mathbb{G}}$ are  its two oriented versions with different orientations denoted by $+e_j$ and $-e_j$. We denote by $\overline{E}$ the set of edges of $\overline{\mathbb{G}}$ and by $\phi$ the covering map.
\begin{center}
 \includegraphics[scale=0.9]{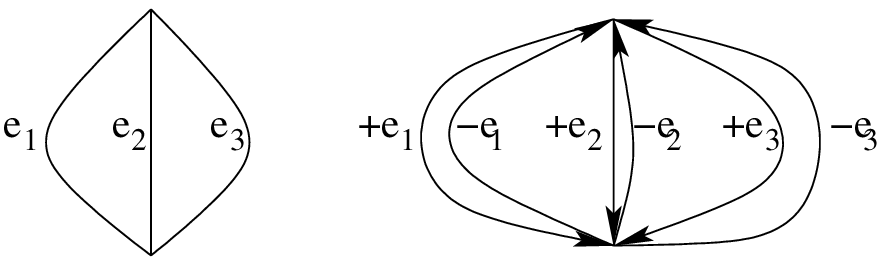}
\end{center}

\begin{definition}
A \textbf{geodesic} on $\mathbb{G}$ is the $\phi$-image of a sequence of $(\overline{e}_n)_{n\in\mathbb{Z}}$ such that for each $n$, we have $\overline{e}_{n+1}$ follows $\overline{e}_n$ and $\overline{e}_{n+1}\neq-\overline{e}_n$. We call one such sequence $(\overline{e}_n)_{n\in\mathbb{Z}}$ an \textbf{oriented geodesic} of $\mathbb{G}$.
\end{definition}
Let $T$ be a positive number. We denote by $R_T(l)$ the set of all oriented closed geodesics with $l$-length less than $T$. 
\begin{definition}
The \textbf{topological entropy} $h(l)$ of $l$ is given by:
\begin{equation}
 h(l)=\lim_{T\rightarrow+\infty}\frac{\log|R_T(l)|}{T}.
\end{equation}
\end{definition}

To define the moduli space of metric graphs, we should avoid the situation where two metrics are obtained from each other by rescaling. To achieve this, there are two ways to renormalize the metrics on $\mathbb{G}$: let ${\rm Vol}(l)=1$ or let $h(l)=1$. Following this idea, we can define two moduli spaces:
\begin{eqnarray}
 \mathcal{M}(\mathbb{G},1)&:=&\{l\in\mathcal{M}(\mathbb{G}):{\rm Vol}(l)=1)\};\\
 \mathcal{M}_1(\mathbb{G})&:=&\{l\in\mathcal{M}(\mathbb{G}):h(l)=1\}.
\end{eqnarray}

We will follow the idea in \cite{SP} and only consider $\mathcal{M}_1(\mathbb{G})$ from now on.
\begin{remark}
 Notice that there is a natural bijection $b:\mathcal{M}(\mathbb{G},1) \rightarrow\mathcal{M}_1(\mathbb{G})$ by sending $l$ to $h(l)l$. 
 
 For closed surface case, these two renormalizations of constant curvature metrics on $S$ are the same. More precisely, the volume of a hyperbolic surface only depends on the topological type of the surface. Meanwhile the entropy of the geodesic flow for a closed hyperbolic surface is always $1$. Therefore if we consider the Teichm\"uller space, we renormalize metrics in the two ways at the same time.
 
 For bordered surface case, we still consider the Teichm\"uller space. This means that we renormalize metrics to have the same volume while the entropy function is not constant.
 \QEDclosed
\end{remark}

\subsection{Weil-Petersson type metric on $\mathcal{M}_1(\mathbb{G})$}
Let $A$ be a $\{0,1\}$-matrix in a size $k\times k$ defined by:
\begin{equation*}
 A(\overline{e},\overline{e}')
 = \left\{ \begin{array}{ll}
1 & \textrm{if $\overline{e}'$
  follows $\overline{e}$ and
  $\overline{e}'\neq-\overline{e}$}\\
0 & \textrm{otherwise}.
\end{array} \right.
\end{equation*}

Then we define
the following subshift
space of finite type :
\begin{equation}
 X'=\{(\overline{e}_n)_{n\in\mathbb{Z}}
 \,:\,\forall n\in\mathbb{Z},\, \overline{e}_n\in\overline{E}
 \,\,{\rm and }\,\, A(\overline{e}_n,\overline{e}_{n+1})=1\},
\end{equation}
associated with a shift operator
$\sigma:X'\rightarrow X'$.

We denote by $C(X',\mathbb{R})$ the space of H\"older continuous function on $X'$. For a function $F\in C(X',\mathbb{R})$, we can define its topological entropy, pressure and its associated equilibrium state in the same way as we described in Section \ref{S}. Again, we are interested in a subspace of $C(X',\mathbb{R})$ consisting of pressure $0$ functions and we denote it by $C_0(X',\mathbb{R})$. By applying the thermodynamic formalism, we can describe its tangent space and the Hessian of pressure function induces a positive semi-definite 2-form on $C(X',\mathbb{R})$.
 
A metric $l$ on $\mathbb{G}$ induces a locally constant function $F_l$ on $X'$ such that:
\begin{equation*}
 F_l((\overline{e}_n)_{n\in\mathbb{Z}})=l(\overline{e}_0).
\end{equation*}

By sending $l\in\mathcal{M}_1(\mathbb{G})$ to $-F_l$, we obtain the thermodynamic map $\mathcal{I}:\mathcal{M}_1(\mathbb{G})\rightarrow C_0(X',\mathbb{R})$.

\begin{theorem}[Sharp-Pollicott \cite{SP}]
 The pullback of the Hessian of pressure function by $\mathcal{I}$ is positive definite and induces a Riemannian metric on $\mathcal{M}_1(\mathbb{G})$.
\end{theorem}

\begin{definition}[Sharp-Pollicott \cite{SP}]
	This Riemannian metric on $\mathcal{M}_1(\mathbb{G})$ is called the \textbf{\textit{Weil-Petersson type metric}}. We denote it by $g'_{wp}$.
\end{definition}
\begin{definition}
 The \textbf{\textit{pressure metric}} $g'_p$ on $\mathcal{M}_1(\mathbb{G})$ is obtained by rescaling $g'_{wp}$ by the function $$\frac{1}{\int F_l\,{\rm d}\mu_{-F_l}}.$$
\end{definition}
With these notation, our main theorem is stated as follows:
\begin{maintheorem}
	For a bordered surface $S$, the metric $g_p$ on $\mathcal{T}(S)$ is incomplete. There exists a graph $\mathbb{G}$ such that a partial completion of $(\mathcal{T}(S),g_p)$ can be given by $(\mathcal{T}(S),g_p)\sqcup(\mathcal{M}_1(\mathbb{G}),g'_p)$.
\end{maintheorem}

\subsection{An alternative symbolic coding for the geodesic flow of $\mathbb{G}$}
In Section \ref{S}, we introduced $X$ a subshift space of finite type which is a coding space for the geodesic flow on $\mathrm{NW}S$. In this part, we show that this coding space can also be used for the geodesic flow on $\mathbb{G}$. Therefore $\mathcal{M}_1(\mathbb{G})$ can be also embedded in $C_0(X,\mathbb{R})$ by the thermodynamic map.

Let us recall the fact that the graph $\mathbb{G}$ is the dual graph of the ideal triangulation $T$ which we chose when we give the definition of pressure metric on $\mathcal{T}(S)$. There we choose a subset of arcs $\{\alpha'_1,...,\alpha'_{2g+r-1}\}$ in $T$ to cut the surface $S$ into a topological disk. Notice that when we cut $S$, we also cut $\mathbb{G}$ as an embedded graph in $S$. By cutting along the same arcs, the resulting graph is simply connected. Then each orbit of the geodesics flow on $\mathbb{G}$ will be cut into infinitely many oriented segments.  The starting points of these oriented segments have the type in $I=\{\alpha_1'^{\pm},...,\alpha_{2g+r-1}'^{\pm}\}$ defined above. Therefore, we can also use $I$ to code the geodesics flow of $\mathbb{G}$.

Given a metric $l$ on $\mathbb{G}$, we would like to consider it as a H\"older continuous function on $X$. Thus we have to decide the cut point on each edge. We will choose the cut point to be the mid point of each edge. Then for each metric $l\in \mathcal{M}_1(\mathbb{G})$, we obtain a locally constant function $F_l\in C(X,\mathbb{R})$. Since the topological entropy of $l$ is $1$, The thermodynamic map is defined by sending $l$ to $-F_l$. By thermodynamic formalism, we can define a pressure metric $g''_p$ on $\mathcal{M}_1(\mathbb{G})$. In fact, we can show that,
\begin{lemma}
 This pressure metric $g''_p$ on $\mathcal{M}_1(\mathbb{G})$ coincides with $g'_p$ defined above. 
\end{lemma}
\begin{proof}
Let us first recall the definition of the renormalized intersection function between two metrics $l_1$ and $l_2$ in $\mathcal{M}_1(\mathbb{G})$:
\begin{equation}
 J(l_1,l_2):=\frac{h(l_2)}{h(l_1)}\lim_{T\rightarrow\infty}
 \left(\frac{1}{|R_T(l_1)|}\left(\sum_{\gamma\in R_T(l_1)}
 \frac{l_2(\gamma)}{l_1(\gamma)}\right)\right).
\end{equation}
We fix the first variable and obtain a function from $\mathcal{M}_1(\mathbb{G})$ to $\mathbb{R}$. The Hessian of this function on the second variable coincides with $g'_p$, as well as $g''_p$. This completes the proof.
\end{proof}

\section{Degeneration of hyperbolic bordered surfaces to a metric graph}
In the introduction, we have describe a natural retraction of a bordered surface $S$ to its fat graph $\mathbb{G}$. Recall that $\mathbb{G}$ is dual to the ideal triangulation $T$ which we used to define the coordinates system for $\mathcal{T}(S)$. In this section, we will describe this retraction in a geometric way: a metric on $\mathbb{G}$ can be viewed as a projective limit of a sequence of hyperbolic metrics on $S$.

The coordinates $(b_1,\dots,b_{3s})$ of $\mathcal{T}(S)$ can be separated into triples where each triple $(b_i,b_j,b_k)$ consists of lengths of three boundary segments contained in one right-angled hexagon $H_n$ in $S\setminus T$. We consider those points in $\mathcal{T}(S)$ whose coordinates are such that every triple $(b_i,b_j,b_k)$ associated to one right-angled hexagon satisfy strict triangular inequalities. These points form an open cone $\mathcal{C}$ in $\mathcal{T}(S)$. 

Let $m=(b_1,\dots,b_{3s})$ be a point in $\mathcal{C}$. We assume that the sum of $b_i$'s is $1$. We consider the line $\{m_\lambda=(\lambda b_1,\dots,\lambda b_{3s})\mid\lambda\in\mathbb{R}^+\}$. We would like to show that its projective limit is a metric on the graph. To do this, we rescale each marked hyperbolic structure $m_\lambda$ by a factor $\lambda^{-2}$ so that the length element is rescaled by $\lambda^{-1}$. We denote the renormalized metric by $m_\lambda'$ which has constant curvature $-\lambda^4$. By the cosine rule for right-angled hexagons in the hyperbolic plane, since a triple $b_i$, $b_j$ and $b_k$ satisfy strict triangle inequalities, the $m_\lambda$-lengths of their opposite sides which are arcs in $T$ will go to $0$ exponentially fast when $\lambda$ goes to infinity. As the $m'_\lambda$-curvature goes to $-\infty$ when $\lambda$ goes to $\infty$, we can see that in the limit, each hexagon converges to a metric graph with one vertex and three branches in the sense of Gromov-Hausdorff. 
\begin{center}
 \includegraphics[scale=0.9]{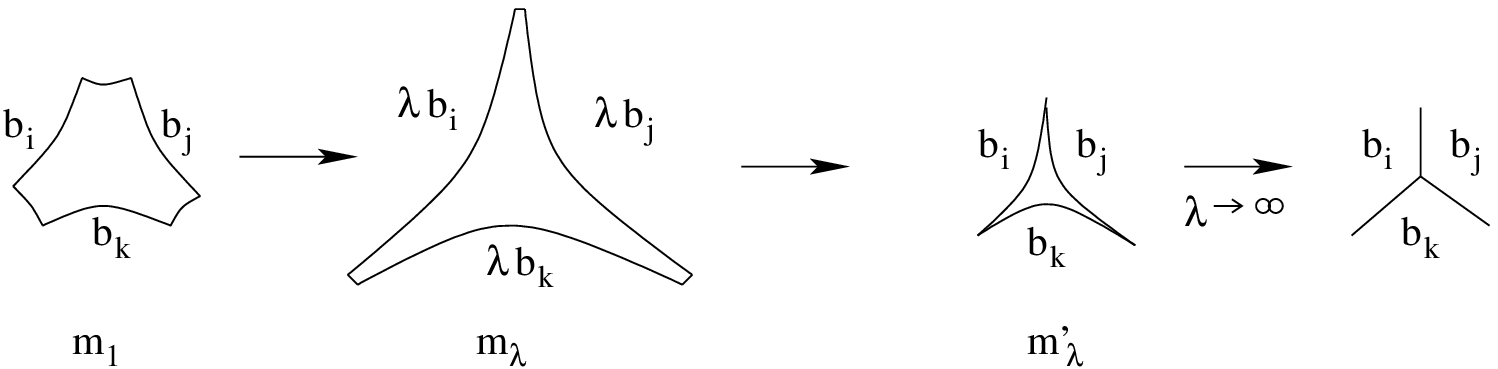}
\end{center}

Now consider the other hexagons $H^c_n$ whose sides does not contain $b_i$'s. By our notation given in Section \ref{S}, their sides on the boundary are denoted by $b^c_i$. When we change $m$ to $m_\lambda$, the length $b^c_i$ will not change to $\lambda b^c_i$. But we can use the cosine rule for right-angled hexagons to study their asymptotic behavior when $\lambda$ goes to infinity. Since all arcs in the ideal triangulation have their lengths converge to $0$ as $\lambda$ goes to infinity, we can conclude that a triple $b^c_i$, $b^c_j$ and $b^c_k$ contained in a right-angled hexagon $H^c_n$ will go to infinity. Moreover they will satisfy triangular inequality when $\lambda$ is big enough. To be more precisely, Let $H_1$, $H_2$ and $H_3$ be its three neighbors. Let $b_{q1}$, $b_{q2}$ and $b_{q3}$ be the sides of $H_q$ lying on the boundary of surface. Let $\alpha_q$ be the arc in $T$ separating $H^c_n$ from $H_q$. 
\begin{center}
 \includegraphics[scale=1.0]{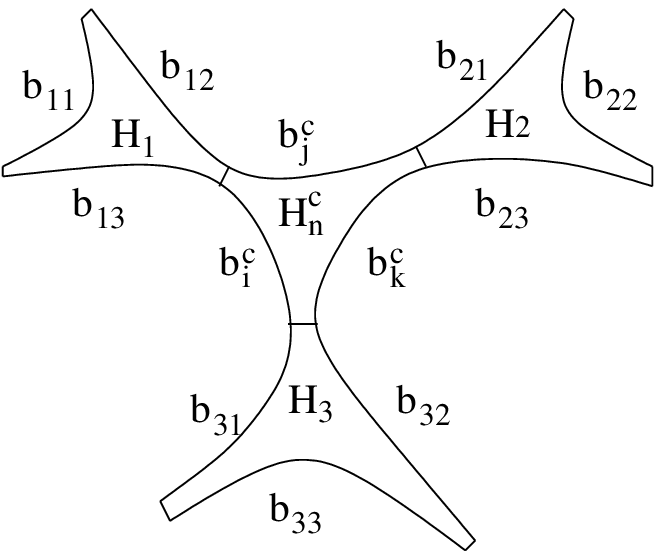}
\end{center}

When $\lambda$ is big enough, the length of $\alpha_1(\lambda)$ will be equal to $\sqrt{2}\exp(\lambda(b_{11}-b_{12}-b_{13})/2)$ plus $O(\exp(\lambda(b_{11}-b_{12}-b_{13})))$. This implies that $(b^c_k(\lambda))-b^c_i(\lambda)-b^c_j(\lambda))/\lambda$ converges to $b_{11}-b_{12}-b_{13}$ exponentially fast as $\lambda$ goes to infinity. The same argument gives the asymptotic behavior of $(b^c_j(\lambda)-b^c_k(\lambda)-b^c_i(\lambda))/\lambda$ and $(b^c_i(\lambda)-b^c_j(\lambda)-b^c_k(\lambda))/\lambda$. Therefore we get the following asymptotic behavior:
\begin{eqnarray*}
 &&\frac{b^c_i(\lambda)}{\lambda}\rightarrow\frac{1}{2}(b_{12}+b_{13}-b_{11}+b_{31}+b_{32}-b_{33}),\\
 &&\frac{b^c_j(\lambda)}{\lambda}\rightarrow\frac{1}{2}(b_{12}+b_{13}-b_{11}+b_{21}+b_{23}-b_{22}),\\
 &&\frac{b^c_k(\lambda)}{\lambda}\rightarrow\frac{1}{2}(b_{21}+b_{23}-b_{22}+b_{31}+b_{32}-b_{33}),
\end{eqnarray*}
as $\lambda$ goes to infinity. Therefore the renormalization of $H^c_n$ will converge to a graph with one vertex and three branches attaching to it.

When we glue all hexagons together, we can see that the surface $(S, m'_\lambda)$ degenerates to $\mathbb{G}$ equipped with a metric determined by $(b_1,\dots,b_{3s})$.

\begin{theorem}\label{sol}
	When $\lambda$ goes to infinity, the project limit of the cone $\mathcal{C}$ is $\mathcal{M}_1(\mathbb{G})$. The paths starting from different points in $\{(b_1,\dots,b_{3s})\mid\sum^{3s}_{i=1}b_i=1\}\cap\mathcal{C}$ will have different projective limit.
\end{theorem}
\begin{proof}
The above discussion tells us the limit of the path $m'_\lambda$ is certain metric $l$ on $\mathbb{G}$. To obtain a metric with entropy $1$, we need to renormalize the metric by a constant factor $h(l)$. This gives one direction which says that the projective limit of $\mathcal{C}$ is contained in $\mathcal{M}_1(\mathbb{G})$. 

Now we prove the rest of the statement. Consider a metric $l\in\mathbb{G}$. We would like to show that there exists $(b_1,\dots,b_{3s})$ such that the projective limit of the path is $l$. Moreover, the length $b_i$ can be uniquely determined by the $l$-lengths of edges of $\mathbb{G}$. This can be done by considering the following observation. The degeneration of the surface that we described above also implies that the arcs in $T$ the ideal triangulation will degenerate to the mid points of edges of $\mathbb{G}$. This induces the relation between $(b_1,\dots,b_{3s})$ and $l$ which tells us there is a one to one correspondence between $\{(b_1,\dots,b_{3s})\mid \sum^{3s}_{i=1}b_i=1\}\cap\mathcal{C}$ and $\mathcal{M}_1(\mathbb{G})$.
\end{proof}

\section{Proof of main results}
\subsection{Proof of Main Theorem}
First we explain that the renormalization that we have made on the hyperbolic metric on surface for describing the convergence will not change the image of $\mathcal{T}(S)$ under the thermodynamic map $\mathcal{I}$. This is because the function in $\mathcal{I}$-image has the form $-h_FF$. Rescaling the metric by a constant factor means rescaling the corresponding H\"older continuous function $F$ by a constant factor. This will change the entropy by the inverse of the same factor which makes the product function $-h_FF$ unchanged. Therefore during the whole proof, we will consider the function $F_\lambda\in C_0(X,\mathbb{R})$ induced by the metric $m_\lambda'$ on $S$.

By Theorem \ref{sol}, we can define the path of degeneration $r:[0,1]\rightarrow\mathcal{I}(\mathcal{M}_1(\mathbb{G})\sqcup\mathcal{T}(S))$ such that the parameter $t=1/\lambda$ and $r(t)=-h_tF_t$ where $F_t=F_\lambda$. Let $m'_t=m'_\lambda$. As a convention, $r(0)$ represents the limit metric on $\mathbb{G}$. 

\begin{lemma}
 The right derivative $r'(0)$ is a $0$ function.
\end{lemma}
\begin{proof}
Let $\gamma$ be a geodesic entirely contained in $S$. Under the partition of $\mathrm{NW}S$, it is cut into segments. Let $\beta$ be one of them. Depending on the positions of its starting and ending points, the segment $\beta$ may pass through several right-angled hexagons in $S\setminus T$, which induces a partition of $\beta$. Each part is the intersection between $\beta$ and one right-angled hexagon. Denote by $\{\beta_1,\dots,\beta_N\}$ be the subsegment of $\beta$. Let $l_t(\beta_n)$ be the length of $\beta_n$ with respect to $m'_t$. It is clear that:
\begin{equation}
 F_t(\beta)=\sum^N_{n=1}l_t(\beta_n),
\end{equation}
and 
\begin{equation}
 F'_t(\beta)=\sum^N_{n=1}l'_t(\beta_n),
\end{equation}
To show $F'_0(\beta)=0$, it is enough to show $l'_0(\beta_n)=$ for all $n$. Notice that $\beta_n$ is the intersection of $\beta$ with either $H_n$ or $H^c_n$. We discuss these two cases separately.

Assume that $\beta_n$ is contained in $H_{n}$ which is determined by $b_{n_1}$, $b_{n_2}$ and $b_{n_3}$. When $t=0$, $\beta_n$ will be coincide with one of $b_{n_1}$, $b_{n_2}$ and $b_{n_3}$. Assume that it is $b_{n_1}$. Then $l_0(\beta_n)=b_{n_1}$. Therefore we have the following inequality:
\begin{equation}\label{est}
 l_0(\beta_n)\le l_t(\beta_n)\le l_0(\beta_n)+2 {\rm max}\{t\alpha_j(t)\mid j\in\{1,\dots,3s\}\},
\end{equation}
where $\alpha_j(t)$ is the $m_t$-length of the corresponding arc in $T$. By the estimation of $\alpha_j(t)$, for $t$ small enough, there exists a constant $c$ such that the hyperbolic length of $\alpha_j(t)$ is at most $2\exp(-c/2t)$ for any $j$. Thus the $m'_t$-length of $\alpha_j(t)$ is at most $2t\exp{(-c/2t)}$. Then the Inequality (\ref{est}) implies that: 
\begin{equation}\label{est1}
 \left|l_t(\beta_n)-l_0(\beta_n)\right|\le4te^{-\frac{c}{2t}}.
\end{equation}
We conclude that $l'_0(\beta_n)$ is $0$ when $t=0$.

Now assume that $\beta_n$ is contained in $H^c_n$ determined by $b^c_{n_1}$, $b^c_{n_2}$ and $b^c_{n_3}$. We have studied the asymptotic behavior of the lengths $tb^c_i(t)$, $tb^c_j(t)$ and $tb^c_k(t)$ in the proof of Theorem \ref{sol} as $t$ goes to $0$. By these estimation, we can conclude the same inequality for $l_t(\beta_n)$ as Inequality (\ref{est1}). This implies $l'_0(\beta_n)=0$. Taking the sum of $l'_0(\beta_n)$ over $n$, we obtain $F'_0(\beta)=0$.

A simple transformation of Inequality (\ref{est1}) is:
\begin{equation}
 l_0(\beta_n)(1-4\frac{te^{-\frac{c}{2t}}}{c'})\le l_t(\beta_n)\le l_0(\beta_n)(1+4\frac{te^{-\frac{c}{2t}}}{c'}),
\end{equation}
where $c'=\min\{b_1,\dots,b_{3s},b^c_1,\dots,b^c_{3s}\}$. By the definition of the topological entropy, we obtain the following inequality:
\begin{equation}
	h_0(1+4\frac{te^{-\frac{c}{2t}}}{c'})^{-1}\le h_t\le h_0 (1-4\frac{te^{-\frac{c}{2t}}}{c'})^{-1}.
\end{equation}
This implies the differential of $h_t$ at $0$ is also $0$. Since $h_t$ converges to $h_0$ and $F_t$ converges to $F_0$ as $t$ goes to $0$, we conclude that the function $r'(0)=(-h_0F_0)'$ is a $0$ function.
\end{proof}
This implies that the length of the path $r$ with respect to pressure metric is finite. Therefore the pressure metric on $\mathcal{T}(S)$ is incomplete.

\subsection{Proofs of Corollary \ref{WP}}
To compare the Weil-Petersson metric and the pressure metric, we would like to compare their completion. In particular, we consider their completion for the cone $\mathcal{C}$.

Recall that the Weil-Petersson metric on $\mathcal{T}(S)$ is considered as a restriction of Weil-Petersson metric in $\mathcal{T}(\mathrm{D}S)$. One observation is that the double of the ideal triangulation $T$ is a pair of pants decomposition of $\mathrm{D}S$. Therefore the path that we described in the cone $\mathcal{C}$ above for studying the incompleteness of pressure metric is a maximal pinching path for $\mathcal{T}(\mathrm{D}S)$, and moreover they all pinch the same pair of pants decomposition coming from doubling $T$. Therefore by Masur's result in \cite{Ma1} the completion for the Weil-Petersson metric in $\mathcal{T}(S)$ for this end of cone $\mathcal{C}$ is a point.

On the other hand, the main theorem tells us that the completion of the cone $\mathcal{C}$ with respect to the pressure metric is $\mathcal{M}_1(\mathbb{G})$ which has dimension strictly higher than $0$. By the uniqueness of the completion of a metric, we conclude that the pressure metric is not equivalent to the Weil-Petersson metric.

\bibliographystyle{plain}
\bibliography{thesis}

\end{document}